\documentclass[11pt,reqno,oneside]{amsart}
\usepackage[colorlinks,
linkcolor=blue,
anchorcolor=blue,
citecolor=blue,
]{hyperref}
\usepackage{slashed}
\usepackage{amscd}
\usepackage{amsmath}
\usepackage{latexsym}
\usepackage{amsfonts}
\usepackage{amssymb}
\usepackage{amsthm}
\usepackage{graphicx}
\usepackage{verbatim}
\usepackage{mathrsfs}
\usepackage{enumerate}
\usepackage{fullpage}
\usepackage{xcolor}
\usepackage{colonequals}
\usepackage{amsthm}
\usepackage{hyperref}

\newtheorem{theorem}{Theorem}[section]
\newtheorem{corollary}[theorem]{Corollary}
\newtheorem{lemma}[theorem]{Lemma}
\newtheorem{proposition}[theorem]{Proposition}
 \theoremstyle{definition}

	\newtheorem{remark}{Remark}[section]
	\newtheorem{conjecture}{Conjecture}[section]

\numberwithin{equation}{section}

\begin{document}
	
\title{An improved upper bound for the second eigenvalue on tori}

\author{Fan Kang}
\address{Yau Mathematical Sciences Center, Tsinghua University, 100084, Beijing, China}
\email{fankangmath@163.com}
	
\begin{abstract}
In this paper, we study the maximization problem of the second non-zero Laplace eigenvalue $\lambda_2(T,g)$ on a torus $T$, among all unit-area metrics in a fixed conformal class. First, we obtain a new upper bound for $\lambda_2(T_{a,b},g)$ on any flat torus $T_{a, b}$ with $(a, b)\in \mathbb{R}^2$. Our bound improves the general estimate $\lambda_2(T_{a, b},g)\le 4A_c(T_{a, b}, [g])$ obtained in \cite{KS24, EG25} in the case of the torus. As applications, we derive a uniform upper bound $\lambda_2(T,g)< \frac{16\pi^2}{\sqrt{3}}$  for any torus $T$ and any metric $g$, and reduce the Kao-Lai-Osting conjecture to proving an upper bound for $\lambda_2(T_{a,b},g)$ on the specific family of flat tori $T_{a,b}$ with $0\leq a\leq \frac12$ and  $\sqrt{1-a^2}\leq b\leq 1.76$.
\end{abstract}

\maketitle
	
\section{Introduction} 

Let $(M, g)$ be a closed surface endowed with a Riemannian metric $g$, and let $\Delta_g$ denote the Laplace-Beltrami operator on $(M,g)$. It is well known that the spectrum of $\Delta_g$ satisfies
\[
 0=\lambda_{0}(M,g)<\lambda_{1}(M,g)\leq\lambda_{2}(M,g)\leq\dots\leq\lambda_{k}(M,g)\leq\dots\to+\infty,
\]
where each eigenvalue is repeated according to its multiplicity. The normalized eigenvalue functional of $g$ is defined by
\[  
\bar{\lambda}_k(M,g)=\lambda_k(M,g)\operatorname{Area}(M,g),
\]
which is a scale-invariant quantity. For $k=1$, the Yang-Yau inequality~\cite{YY80} gives $\bar{\lambda}_1(M,g)\leq 8\pi\left(\gamma+1\right)$, where $\gamma$ is the genus of $M$. On the other hand, for general $k\geq 1$, Korevaar~\cite{K93} proved that there exists a constant $C(M)$ depending only on the topology of $M$ such that $\bar{\lambda}_k(M,g)\leq C(M)k$. 
Moreover, given a conformal class $[g]=\{\omega g: \omega>0\}$ on $M$, one can define the conformal eigenvalue of $(M,[g])$ by
\[
\Lambda_k(M, [g])=\sup_{g'\in [g]}\bar{\lambda}_k(M,g'),
\]
and the topological eigenvalue of $M$ by
\[
\Lambda_k(M)=\sup_{[g]}\Lambda_k(M,[g]),
\]
where the supremum is taken over all conformal classes $[g]$ on $M$. Such notions were first introduced by Colbois-El~Soufi~\cite{CE03}.
In this paper, we focus on $\Lambda_2$ on tori.

Below, we recall several classical and recent results on $\Lambda_k(M)$ for various surfaces $M$. 
For the sphere, Hersch~\cite{H70} proved that $\Lambda_1(S^2)=8\pi$, with equality attained only by the round metric of constant curvature. Karpukhin-Nadirashvili-Penskoi-Polterovich~\cite{KNPP17} later showed that $\Lambda_k(S^2)=8\pi k$ for all $k\geq 1$, and that the supremum is obtained on a sequence of metrics converging to a union of $k$ touching round spheres. For the real projective plane, it is known from Li-Yau~\cite{LY82} that $\Lambda_1(\mathbb{RP}^2)=12\pi$, with equality attained only by the round metric of constant curvature. For general $k\geq 1$, Karpukhin~\cite{K21} showed that $\Lambda_k(\mathbb{RP}^2)=4\pi(2k+1)$, and that the supremum is achieved on a sequence of metrics converging to a union of $k-1$ identical round spheres and a standard projective plane touching each other, with area ratio $3\!:\!2$ between the projective plane and each sphere. Next, considering the torus, $\Lambda_1(T)=\frac{8\pi^2}{\sqrt{3}}$ is given by Nadirashvili~\cite{N96}, and the flat equilateral metric is the only maximizer; see also~\cite{CKM19, K25}. Furthermore, on the Klein bottle, El~Soufi-Giacomini-Jazar~\cite{EGJ06} and Jakobson-Nadirashvili-Polterovich~\cite{JNP06} showed that the maximizer of $\Lambda_1(\mathbb{KL}^2)$ is the metric induced by the unique minimal immersion into $S^n$ via the first eigenfunctions; see also~\cite{CKM19, N96}. Finally, it is worth mentioning that $\Lambda_1(\Sigma_2) = 16\pi$ was established by Nayatani-Shoda~\cite{NS19} for surfaces of genus two.
	
Considering a fixed conformal class, the value of $\Lambda_k(M,[g])$ is known in several cases. El~Soufi-Ilias-Ros~\cite{EIR97} and El~Soufi-Ilias~\cite{EI02} determined the exact value of $\Lambda_1(T_{a,b},[g_{a,b}])$ for the torus $T_{a,b}=\mathbb{R}^2/\mathbb{Z}(1,0)\oplus\mathbb{Z}(a,b)$ with $a^2+b^2=1$. Besides, these works also provided an upper bound for $\bar{\lambda}_1$ for all other tori. On the other hand, numerical computations for $\Lambda_k(T,[g])$ with $k\geq 1$ were obtained by Kao-Lai-Osting~\cite{KLO17} using finite element and spectral methods. More generally, for any closed surface $(M,g)$, Karpukhin-Stern~\cite{KS24} and Eddaoudi-Girouard~\cite{EG25} independently proved that 
\begin{equation}\label{old Lambda-2}
\Lambda_2(M, [g])\leq 4A_c(M, [g]),
\end{equation}
where $A_c(M, [g])$ denotes the conformal area. 

In this paper, we focus on the torus. The conformal area of the torus was determined explicitly by Li-Yau~\cite{LY82}, Montiel-Ros~\cite{MR86}, and Bryant~\cite{B15}. Combining these results with Nadirashvili's folding method~\cite{N02}, we obtain a new upper bound for the second conformal eigenvalue of the torus. This bound improves the general estimate~\eqref{old Lambda-2} in the case of tori. To state our result, we first recall the parametrization of flat tori.
	
Let $T$ be a torus equipped with a Riemannian metric $g$. By the uniformization theorem, any torus $(T,g)$ is conformally equivalent to a flat torus $(T_{a, b},\, g_{a, b})$ for some $(a, b)\in \mathbb{R}^2$, 
where 
\[
T_{a, b}=\mathbb{R}^2/\Gamma_{a, b}\quad \text{with}\quad \Gamma_{a,b}=\mathbb{Z}(1,0)\oplus\mathbb{Z}(a,b),
\]
and $g_{a, b}$ is the metric on $\mathbb{R}^2/\Gamma_{a, b}$ induced by the Euclidean metric on $\mathbb{R}^2$. In addition, up to isometry and dilations, there is a one-to-one correspondence between the moduli space of flat tori $T_{a,b}$ 
and the fundamental region 
\[
\mathscr{M}=\Big\{(\tilde a, \tilde b)\in\mathbb{R}^2:\ 0\leq \tilde a \leq \tfrac12,\ \tilde b\geq\sqrt{1-\tilde a^2}\Big\}.
\]
Moreover, it is well known that the eigenvalues of the Laplacian on $(T_{a,b},\, g_{a, b})$ are given by
\[
\lambda^{a,b}_{pq}=4\pi^2\Big\{q^2+\left(\frac{p-aq}{b}\right)^2\Big\}
\]
for $(p,q)\in\big\{(p,q)\in\mathbb{Z}\times\mathbb{Z}: q\geq 0, \text{ or } q=0 \text{ and } p\geq 0\big\}$, and the corresponding eigenspaces are spanned by
\[
\begin{aligned}
f^{a,b}_{pq}(x,y)=\cos{2\pi\big(qx+\tfrac{p-aq}{b} y\big)}\quad \text{and}\quad 
g^{a,b}_{pq}(x,y)=\sin{2\pi \big(qx+\tfrac{p-aq}{b} y\big)}.
\end{aligned}  
\]

Our main result is the following.

\begin{theorem}\label{main thm}
Let $(a,b)\in\mathscr{M}$ and let $g$ be a Riemannian metric on $T_{a,b}$ conformal to the flat metric $g_{a,b}$. Then
\begin{equation}\label{main estimate}
\bar{\lambda}_2(T_{a,b},g)< \frac{16\pi^2}{3\sqrt{6}\,b}\frac{\sqrt{2+a^2+b^2+S}}{a^2+b^2+S}\left(3(a^2+b^2)+S \right),
\end{equation}
where $S=\sqrt{(a^2+b^2)(8+a^2+b^2)}$.
\end{theorem}

As a direct consequence of Theorem~\ref{main thm}, we obtain the following uniform bound.

\begin{corollary}\label{coro:uniform_bound}
For any torus $T$, the second topological eigenvalue satisfies 
\[
\Lambda_2(T) < \frac{16\pi^2}{\sqrt{3}}.
\]
\end{corollary}

\begin{remark}
By the results of Li-Yau~\cite{LY82}, Montiel-Ros~\cite{MR86} and Bryant~\cite{B15}, the conformal area of $T_{a,b}$ is given by
\[
A_c(T_{a,b},[g_{a,b}]) = \frac{4\pi^2 b}{1 + b^2 + a^2 - a}
\]
for all $(a,b)\in\mathscr{M}$ satisfying $a^2 + b^2 - a \leq 2$.
Then, combining Theorem~\ref{main thm} with a direct computation yields
\begin{equation}\label{new bdd}
\Lambda_2(T_{a,b}, [g_{a,b}]) < 2\sqrt{\tfrac23}\;
\frac{(1 + A - a)\bigl(A + \frac{S}{3}\bigr) B}{b^{2}(A + S)}\;
A_c(T_{a,b},[g_{a,b}]),
\end{equation}
where $A = a^{2}+b^{2}$ and $B = \sqrt{2+A+S}$. In particular, 
\[
2\sqrt{\tfrac23}\;
 \frac{(1 + A - a)\bigl(A + \frac{S}{3}\bigr) B}{b^{2}(A + S)}<4\quad \text{whenever}\quad a^2+b^2>1.
\]
Therefore, the bound in~\eqref{new bdd} improves the bound in~\eqref{old Lambda-2}.
\end{remark}

In relation to Theorem~\ref{main thm}, we now recall the Kao-Lai-Osting conjecture for the case $k=2$.

\begin{conjecture}[Kao-Lai-Osting~\cite{KLO17}]\label{Conj-1}
The second topological eigenvalue of the torus satisfies
\begin{equation}\label{eq:conj_T2}
\Lambda_2(T) = \frac{8\pi^2}{\sqrt{3}}+8\pi,
\end{equation}
and the supremum is attained by a sequence of surfaces degenerating to a union of an equilateral flat torus and a round sphere.
\end{conjecture}

As another consequence of Theorem~\ref{main thm}, Conjecture~\ref{Conj-1} reduces to the following statement.

\begin{corollary}\label{main coro}
The Kao-Lai-Osting conjecture holds provided that for all $(a,b) \in \mathscr{M}$ with $b \leq 1.76$, and for all metrics $g$ on $T_{a,b}$ conformal to $g_{a,b}$,
\[
\bar\lambda_2(T_{a,b}, g)  \leq \frac{8\pi^2}{\sqrt{3}} + 8\pi.
\]
\end{corollary}

The paper is organized as follows. Section~2 develops the tools for proving Theorem~\ref{main thm}. In Subsection~2.1, we recall and adapt Nadirashvili's method of constructing trial functions for $\lambda_2$ via conformal transformations and folding maps. Besides, a family of embeddings $\Psi_{a,b}:T_{a,b}\to S^3$ is introduced and the supremum of their Dirichlet energy over conformal transformations (Theorem~\ref{energy functional}) is computed in Subsection~2.2. Finally, in Section~3, we present the proofs of Theorem~\ref{main thm}, and Corollaries~\ref{coro:uniform_bound} and~\ref{main coro}.

\section{Preliminaries}

Let $(M,g)$ be a Riemannian surface, and let $\lambda_k(M,g)$ be the $k$-th eigenvalue of the Laplace-Beltrami operator $\Delta_g$. For any integer $k\geq 1$, we denote by $V_k$ the subspace of $H^1(M)$ spanned by the first $k$ eigenfunctions of $\Delta_g$, and by $V^{\perp}_{k}$ the $L^2$-orthogonal complement of $V_k$ in $L^2(M)$. The following characterization holds:
\[
\lambda_{k}(M,g)=\inf_{f\in V^{\perp}_{k}\setminus \{0\} }\frac{\int_{M}|\nabla f|^2\,dv_g}{\int_{M}f^2\,dv_g}.
\]
Any nonzero function $f\in V^{\perp}_{k}$ is called a trial function for $\lambda_k$. To prove Theorem~\ref{main thm}, we construct a trial function $f\in V^{\perp}_{2}$ which yields an upper bound for $\lambda_2$. A well-known method for constructing trial functions was introduced by Nadirashvili~\cite{N02} to prove a sharp upper bound for $\lambda_2(S^2)$, and was subsequently used in~\cite{GL21, GNP09, P14, K22, EG25}. In the next subsection, we recall and adapt this method to construct our trial functions for $\lambda_2(M,g)$.

\subsection{Construction of trial functions for $\lambda_2(M,g)$}

Let $S^n$ denote the $n$-dimensional unit sphere, $D^{n+1}$ the open unit ball in $\mathbb{R}^{n+1}$, and $G(n)$ the conformal group of $S^n$. For each $\xi\in D^{n+1}$, we consider the map $\phi_{\xi}:S^n\to S^n$, given by
\begin{equation}\label{conformal-map}
\phi_{\xi}(p)=\frac{p+\big(\beta\langle p, \xi\rangle+\alpha\big)\xi}{\alpha\big(\langle p,\xi\rangle+1\big)},\quad  p\in S^n,
\end{equation}
where $\alpha=(1-|\xi|^2)^{-1/2},\,\beta=(\alpha-1)|\xi|^{-2}$, and $\langle \ ,\ \rangle$ denotes the usual inner product on $\mathbb{R}^{n+1}$. The map $\phi_\xi$ is a conformal transformation of $S^n$, and each element of $G(n)$ can be expressed as $O\circ\phi_\xi$ for some orthogonal transformation $O$.
 
Let $\mu$ be a finite Borel measure on $S^n\subset\mathbb{R}^{n+1}$, and let $\pi: S^n\to\mathbb{R}^{n+1}$ denote the canonical embedding. The center of mass of $\mu$ is defined by
 \[
 \frac{1}{\mu(S^n)}\int_{S^n}\pi\,d\mu\in D^{n+1}.
 \]
We say that $\phi_\xi\in G(n)$ conformally renormalizes the center of mass of $\mu$ if the push-forward measure $\nu=(\phi_\xi)_{*}\mu$ satisfies
 \[
 \int_{S^n}\pi\,d\nu=0,
 \]
equivalently,
\[
\int_{S^n}\phi_\xi\,d\mu=0.
\]

Analogous to Hersch’s lemma~\cite{H70}, Langesen proved the following result.

\begin{lemma}[Langesen\textup{\cite{L21}}, Corollary 5]\label{conformal-lemma} 
Let $\mu$ be a Borel measure on the sphere $S^{n}$ satisfying $0<\mu(S^{n})<\infty$. If for all $y \in S^{n}$,
\begin{equation}\label{center of mass}
\mu(\{y\})<\tfrac12\mu(S^{n}),
\end{equation}
then there exists a unique point $\xi=\xi(\mu)\in D^{n+1}$ such that
\[
\int_{S^{n}}\phi_{\xi}\,d\mu=0.
\]
Moreover, $\xi(\mu)$ depends continuously on the measure $\mu$. That is, if $\mu_k$, $\mu$ satisfy~\eqref{center of mass} and $\mu_{k}\to\mu$ weakly, then $\xi(\mu_{k})\to\xi(\mu)$ as $k\to\infty$.
\end{lemma}

We now apply Lemma~\ref{conformal-lemma} to the measure on $S^n$ induced by an embedding of $M$. Let $M$ be a compact Riemann surface admitting an embedding $\Psi: M\to S^n$, and let $g$ be a Riemannian metric on $M$ with volume measure $dv_g$. By Lemma~\ref{conformal-lemma}, there exists a unique point $\xi\in D^{n+1}$ such that $\int_{M}\phi_{\xi}\circ\Psi\,dv_g=0$. Set $f=\phi_{\xi}\circ\Psi$, and let $f_1$ be an eigenfunction associated with $\lambda_1(M,g)$. Then $\int_M f\,dv_g=0$. If additionally $\int_{M} f f_{1} \, dv_g = 0$, then $f$ is a trial function for $\lambda_2(M,g)$. Otherwise, we introduce the following folded measure construction on $S^n$.
 
Let $\mathscr{C}$ denote the set of all spherical caps on $S^n$, and fix $C\in \mathscr{C}$. Set $C^*= S^n\setminus C$. For each $C\in \mathscr{C}$, there exists a unique conformal reflection $\tau_C: S^n\to S^n$ that reverses the orientation of $S^n$ and acts as the identity on $\partial C$. Moreover, we have $\tau_C=\tau_{C^*}$ and $\tau_C(C)=C^*$. Define the folding map $F_C:S^n\to C$ by
\[
F_C(x) = 
\begin{cases}
x & \text{if } x \in C, \\
\tau_C(x) & \text{if } x \in C^*. 
\end{cases}
\]
Consider the push-forward measure $\nu=(F_C)_*(f_*\,dv_g)$. Since $\nu(\{y\})=0$ for all $y\in S^n$, Lemma~\ref{conformal-lemma} implies the existence of a unique renormalization point $\xi_C\in D^{n+1}$ such that 
\[
\int_{S^n}\phi_{\xi_{C}}\,d\nu=\int_{M}\phi_{\xi_C}\circ F_C\circ f\,dv_g=0.
\]
Among the  family of maps $\phi_{\xi_C}\circ F_C\circ f$, with $C\in\mathscr{C}$, we will show that there exists a cap $C$ such that 
\[
\int_{M}(\phi_{\xi_C}\circ F_C\circ f)\cdot f_1\,dv_g=0.
\]
 
In summary, we have the following lemma.

\begin{lemma}\label{main-lemma}
Let $(M,g)$ be a compact Riemann surface and let $f_{1}\in \mathcal{C}^{\infty}(M)$ be an eigenfunction associated with the first eigenvalue $\lambda_{1}(M,g)$. For any embedding $\Psi: M\to S^{n}$ and for each spherical cap $C\subset S^{n}$, there exists a unique $\xi_{C}\in D^{n+1}$ such that
\[
\int_{M}\phi_{\xi_{C}}\circ F_{C}\circ\Psi\,dv_g=0.
\]
Moreover, if $\int_{M}\Psi f_{1}\,dv_g\neq 0$, then there exists a cap $C\subset S^n$ such that
\[
\int_{M}(\phi_{\xi_{C}}\circ F_{C}\circ \Psi)\cdot f_{1}\,dv_g=0.
\]
\end{lemma}

\begin{proof}
The statement follows from the argument of Lemma 2.8 in~\cite{EG25}, with conformal immersions replaced by embeddings. For completeness, we include the adapted proof.
    
The first assertion was established in the preceding paragraph. We now prove the existence of a cap $C$. To this end, recall the standard parametrization of spherical caps $C \subset S^n$ by pairs $(p,t) \in S^n \times (-1,1)$. Let
\[
C(p,0) = \{ x \in S^n : \langle x, p \rangle > 0 \}
\]
be the hemisphere centered at $p$. Then the spherical cap with parameter $t$ centered at $p$ is defined by
\[
C(p,t) = \phi_{-tp}(C(p,0)).
\]
Arguing by contradiction, suppose that for all spherical caps $C=C(p,t)$, 
\[
 h(p,t)=\int_{M}(\phi_{\xi_C}\circ F_C\circ\Psi)\cdot f_1\,dv_g\neq 0.
\]
Define the normalization map $H:S^n\times [0,1)\to S^n$ by
\[
H(p,t)=\frac{h(p,t)}{\|h(p,t)\|}.
\]
Since the renormalization point $\xi(\mu)$ depends continuously on the measure $\mu$, the map $H$ is continuous. Moreover, by the definition of $C(p,t)$, as $t\to 1$ we have $C(p,t)\to S^n$, hence $H$ extends continuously to the boundary by
\[
H(p,1)=\lim_{t\to 1}H(p,t)=\frac{\int_{M}\Psi f_1\,dv_g}{\|\int_{M}\Psi f_1\,dv_g\|}.
\]
The right-hand side is independent of $p$, so $H(p,1)$ is constant and therefore has homotopy degree zero.  We next claim the symmetry
\[
H(p,0)=R_p\circ H(-p,0),\quad p\in S^n,
\]
where $R_p$ denotes the reflection across the hyperplane orthogonal to $p$.
By Proposition~2.4 in~\cite{EG25}, 
\[
\xi_{(p,0)}=R_{p}(\xi_{(-p,0)}),\text{  and  } F_{(p,0)}=R_p\circ F_{(-p,0)},
\] 
where $\xi_C=\xi_{(p,t)}$ and $F_{C}=F_{(p,t)}$. These identities imply

\[
\begin{aligned}
h(-p,0)&=
\int_{M}\left(\phi_{\xi_{(-p,0)}}\circ F_{(-p,0)}\circ\Psi\right)f_1\,dv_g\\
    &=\int_{M}\left(\phi_{\xi_{(-p,0)}}\circ R_p\circ F_{(p,0)}\circ\Psi\right)f_1\,dv_g\\
    &=\int_{M}\left(R_p\circ\phi_{R_p(\xi_{(-p,0)})}\circ F_{(p,0)}\circ\Psi\right)f_1\,dv_g\\
    &=R_p\left(\int_{M}\big(\phi_{\xi_{(p,0)}}\circ F_{(p,0)}\circ\Psi\big)f_1\,dv_g\right)\\
    &=R_p\circ h(p,0).
\end{aligned}  
\]
Moreover, since the reflection $R_p$ is an isometry of $\mathbb{R}^{n+1}$, that is, $\|R_{p}(x)\|=\|x\|$ for all $x\in\mathbb{R}^{n+1}$, we have
\[
\|h(-p,0)\|=\|R_p\circ h(p,0)\|=\|h(p,0)\|.
\]
Using the linearity of $R_p$, it follows that
\[
\begin{aligned}
R_p\circ H(p,0)&=\frac{1}{\|h(p,0)\|}R_p\circ h(p,0)\\
     &=\frac{1}{\|h(-p,0)\|}h(-p,0)=H(-p,0).
\end{aligned}  
\]
By~\cite{P14}, this symmetry implies that $H(\cdot,0)$ has nonzero homotopy degree, whereas $H$ extends continuously to $S^n\times[0,1]$ with boundary degree zero. 
This contradiction shows that there exists a cap $C$ such that
\[
\int_M (\phi_{\xi_C}\circ F_C\circ\Psi)\cdot f_1\,dv_g=0.
\]
\end{proof}

\subsection{Dirichlet energy} In this subsection, we consider an embedding of $T$ into $S^n$, and compute the Dirichlet energy of trial functions for $\lambda_2(T,g)$ using the method of the previous subsection. This idea is inspired by El~Soufi-Ilias-Ros~\cite{EIR97} and El~Soufi-Ilias~\cite{EI02}.

Let $M$ be a compact surface. For each branched conformal immersion $\psi:M\to S^n$, we consider the area function $A: G(n)\to \mathbb{R}$, which assigns to each conformal transformation $\phi_\xi$ of $S^n$ the area of the immersion $\phi_\xi\circ\psi$. We denote this by $A(\phi_\xi\circ\psi)$. Since every element of $G(n)$ can be written in the form~\eqref{conformal-map}, 
the function $A$ may be regarded as a function on the unit ball 
$D^{n+1}$. In particular, for each $\xi\in D^{n+1}$,
\[
A(\phi_\xi\circ\psi)=\tfrac12\int_{M}\frac{1-|\xi|^2}{(1-\langle \psi,\xi\rangle)^2}
|\nabla \psi|^2\,dM.
\]
We now recall the definition of conformal area. The conformal area of $M$ associated with $\psi$ is 
\[
A_c(n,\psi)=\sup_{\phi_\xi\in G(n)}A(\phi_\xi\circ\psi).
\]
The $n$-conformal area of $(M,[g])$ is then given by
\[
A_c(n,M,[g])=\inf_{\psi:M\to S^n} A_c(n,\psi),
\]
where the infimum is taken over all conformal immersions $\psi:M\to S^n$.
Finally, the conformal area of $(M,[g])$ is defined as
\[
A_c(M,[g])=\lim_{n\to\infty}A_c(n,M,[g]).
\]

For the torus, the following explicit formula was obtained by Bryant~\cite{B15}.

\begin{lemma}[Bryant\textup{~\cite{B15}}]\label{main lemma}
Let $M$ be a torus conformally equivalent to $T_{b}:=T_{0,b}$, and let $\psi_{b}\colon T_b\to S^3$ be the following immersion:
\[
\psi_{b}(x,y)=\frac{1}{\sqrt{1+b^2}}\Big(b\cos\frac{2\pi y}{b},b\sin\frac{2\pi y}{b}, \cos{2\pi x},\sin{2\pi x}\Big).
\]
Then, for $1\leq b \leq\sqrt{2}$, we have
\begin{equation}\label{case 1}
\sup\limits_{\gamma\in G(3)}A(\gamma\circ\psi_{b})=\frac{4\pi^2 b}{1+b^2};
\end{equation}
for $b>\sqrt{2}$, 
\begin{equation}\label{case 2}
  \sup\limits_{\gamma\in G(3)}A(\gamma\circ\psi_{b})= \frac{8\pi^2\sqrt{b^2+1}}{3\sqrt{3}b}.
\end{equation}
\end{lemma}

\begin{proof}
The proof is due to Bryant. For completeness, we include it here, replacing the second part with a direct hand computation.
Denote $r_1=\frac{b^2}{1+b^2}$, $r_2=\frac{1}{1+b^2}$, $t=\frac{2\pi y}{b}$ and $s=2\pi x$. For any $\gamma=(a_1,\dots,a_4)\in\mathbb{R}^4$ with $|\gamma|^2<1$, we have
\[
A(\gamma\circ\psi_b)=\frac{b}{1+b^2}\int_{0}^{2\pi}\int_{0}^{2\pi}\frac{1-\langle \gamma, \gamma\rangle}{\Big(1-\langle \gamma, \psi_b(s,t)\rangle\Big)^2}\,ds\,dt.
\]
Since $\psi_b$ is equivariant with respect to $2$-torus of rotations in $SO(4)$, one can apply a rotation from this torus to reduce to the cases in which $a_1,a_3\geq 0$ and $a_2=a_4=0$.
Set $\lambda=a_1\sqrt{r_1}$, $\mu=a_3\sqrt{r_2}$. Define
\[
\Omega = \left\{ (\lambda, \mu) \in \mathbb{R}^2 \;\middle|\; 
\frac{\lambda^2}{r_1} + \frac{\mu^2}{r_2} < 1,\ 
\lambda \geq 0,\ \mu \geq 0 \right\}.
\]
By Bryant's inequality $(3.7)$ \cite{B15}, we have
\begin{equation}
    \begin{aligned}
\frac{1+b^2}{4\pi^2b}A(\lambda,\mu)&:= \frac{1+b^2}{4\pi^2b}A(\gamma\circ\psi_b)=\frac{1}{4\pi^2}\Big(1-\frac{\lambda^2}{r_1}-\frac{\mu^2}{r_2}\Big)\int_{0}^{2\pi}\int_{0}^{2\pi}\frac{ds\,dt}{\big(1-\lambda\cos{t}-\mu\cos{s}\big)^2}\\
&\leq \Big(1-\frac{\lambda^2}{r_1}-\frac{\mu^2}{r_2}\Big)\frac{1-(\lambda+\mu)^2+3\lambda\mu}{\Big(1-(\lambda-\mu)^2\Big)^{3/2}\Big(1-(\lambda+\mu)^2\Big)}=:I(\lambda,\mu).
\end{aligned}
\end{equation}
The maximization of $I$ over $\Omega$ can be found in Bryant's paper~\cite{B15}. For the case $b>\sqrt{2}$, we provide an alternative direct computation in the following lemma.
\end{proof}

\begin{lemma}
Let
\[
r_1 = \frac{b^2}{1+b^2}, \quad r_2 = \frac{1}{1+b^2},
\]
\[
\Omega = \left\{ (\lambda, \mu) \in \mathbb{R}^2 \;\middle|\; 
\frac{\lambda^2}{r_1} + \frac{\mu^2}{r_2} < 1,\ 
\lambda \geq 0,\ \mu \geq 0 \right\},
\]
and
\[
I(\lambda, \mu) 
= \frac{\left(1 - (\lambda + \mu)^2 + 3\lambda\mu\right) 
\left(1 - \dfrac{\lambda^2}{r_1} - \dfrac{\mu^2}{r_2}\right)}
{\left(1 - (\lambda - \mu)^2\right)^{3/2} 
\left(1 - (\lambda + \mu)^2\right)}.
\]
Then for $b > \sqrt{2}$,
\[
\sup_{(\lambda, \mu) \in \Omega} I(\lambda, \mu)
=
\frac{2(b^2+1)^{3/2}}{3\sqrt{3}\,b^2}.
\]
\end{lemma}

\begin{proof}
We first show that $I(\lambda,\mu)$ has no interior critical point in $\Omega^\circ$.

Set
\[
E:=1-\frac{\lambda^2}{r_1}-\frac{\mu^2}{r_2},\quad
F:=1-\lambda^2+\lambda\mu-\mu^2,
\]
\[
A:=1-(\lambda-\mu)^2,\quad
B:=1-(\lambda+\mu)^2.
\]
Any interior critical point satisfies
\[
\partial_\lambda \log I=0,\qquad \partial_\mu \log I=0,
\]
which yields
\begin{equation}\label{eq:L}
-\frac{2\lambda/r_1}{E}
+\frac{-2\lambda+\mu}{F}
+\frac{3(\lambda-\mu)}{A}
+\frac{2(\lambda+\mu)}{B}=0,
\end{equation}
\begin{equation}\label{eq:M}
-\frac{2\mu/r_2}{E}
+\frac{\lambda-2\mu}{F}
-\frac{3(\lambda-\mu)}{A}
+\frac{2(\lambda+\mu)}{B}=0.
\end{equation}
Adding and subtracting \eqref{eq:L} and \eqref{eq:M}, we obtain
\begin{equation}\label{eq:S}
\frac{2(\lambda/r_1+\mu/r_2)}{E}
=\frac{3(\lambda+\mu)A}{BF},
\end{equation}
\begin{equation}\label{eq:D}
\frac{2(\lambda/r_1-\mu/r_2)}{E}
=\frac{3(\lambda-\mu)(1-\lambda^2-\mu^2)}{AF}.
\end{equation}
Since $\lambda>0$ in the interior, set
\[
t:=\frac{\mu}{\lambda}>0,\qquad x:=\lambda^2.
\]
Then
\[
A=1-(1-t)^2x,\quad
B=1-(1+t)^2x,\quad
F=1-(1-t+t^2)x.
\]
Dividing \eqref{eq:D} by \eqref{eq:S}, we obtain
\begin{equation}\label{eq:E1}
\frac{1-b^2t}{1+b^2t}
=
\frac{1-t}{1+t}
\frac{(1-(1+t^2)x)(1-(1+t)^2x)}
{(1-(1-t)^2x)^2}.
\end{equation}
Clearing denominators in \eqref{eq:E1}, we obtain a quadratic equation in $x$ with a common factor $2t$.
Dividing by this factor, we define
\[
P_1(t,x):=\frac{1}{2t}\Bigl[(1 - b^2 t)(1 + t)(1-(1-t)^2x)^2
-
(1 + b^2 t)(1 - t)(1-(1+t^2)x)(1-(1+t)^2x)\Bigr].
\]
Similarly, \eqref{eq:S} yields another quadratic equation
\[
P_2(t,x)=0,
\]
where
\[
P_2(t,x)
:=
2(1+b^2)\Bigl(\tfrac1{b^2}+t\Bigr)
(1-(1+t)^2x)(1-(1-t+t^2)x)
-3(1+t)(1-(1-t)^2x)
\Bigl(1-(1+b^2)\Bigl(\tfrac1{b^2}+t^2\Bigr)x\Bigr).
\]

We write
\[
P_1 = \alpha_2 x^2 + \alpha_1 x + \alpha_0,
\qquad
P_2 = \beta_2 x^2 + \beta_1 x + \beta_0,
\]
where
\[
\begin{aligned}
\alpha_2 &= 2b^2 t^4 - b^2 t^3 - b^2 t^2 + b^2 t - b^2
          + t^4 - t^3 + t^2 + t - 2,\\
\alpha_1 &= -b^2 t^2 - b^2 t + 2b^2 - 2t^2 + t + 1,\\
\alpha_0 &= 1 - b^2,
\end{aligned}
\]
and
\[
\begin{aligned}
\beta_2 &= -b^2 t^5 + 5b^2 t^4 + 3b^2 t^3 - b^2 t^2 + 2b^2 t \\
&\quad - t^5 + 7t^4 + 2t^3 + 2t^2 + 7t - 1
+ \frac{2t^4 - t^3 + 3t^2 + 5t - 1}{b^2},\\
\beta_1 &= -b^2 t^3 + b^2 t^2 - 4b^2 t
+ 2t^3 - 6t^2 - 6t + 2
+ \frac{-4t^2 + t - 1}{b^2},\\
\beta_0 &= 2b^2 t - t - 1 + \frac{2}{b^2}.
\end{aligned}
\]
Both $P_1$ and $P_2$ are quadratic in $x$. 
We now combine them so that the $x^2$-terms cancel:
\[
\beta_2 P_1-\alpha_2 P_2.
\]
This gives a linear equation in $x$. Solving for $x$ and substituting the result back into $P_1=0$, we obtain an equation depending only on $t$, denoted by
\[
H(t)=0.
\]
A direct computation shows that
\[
H(t)=4t(t+1)(b^2t-1)^2(b^2t+1)^2Q(t),
\]
where
\[
Q(t)=9b^4t^3-2b^4t^2+b^4t+8b^2t^2+8b^2t+t^2-2t+9.
\]
Thus any interior critical point must satisfy $H(t)=0$. Since for $t>0$ we have
\[
t>0,\qquad t+1>0,\qquad b^2t+1>0,
\]
it follows that $H(t)=0$ implies either $t=1/b^2$ or $Q(t)=0$.

\noindent If $t=1/b^2$, solving $P_2(t,x)=0$ yields
\[
x=\frac{b^4}{(1+b^2)^2},
\]
so that
\[
\lambda=\frac{b^2}{1+b^2},\quad 
\mu=\frac{1}{1+b^2},
\]
which gives $\lambda+\mu=1$, contradicting the interior condition.

\noindent Finally,
\[
Q'(t)=27b^4t^2+2(-2b^4+8b^2+1)t+(b^4+8b^2-2),
\]
whose discriminant is negative for $b>\sqrt{2}$. Hence $Q'(t)>0$ for all $t$, and $Q$ is strictly increasing. Since $Q(0)=9>0$, we have $Q(t)>0$ for all $t>0$.

\noindent Therefore $H(t)\neq 0$ for all $t>0$, a contradiction. Hence no interior critical point exists.

\medskip

Thus the maximum is attained on the boundary of $\Omega$. A direct computation on the boundary yields
\[
\sup_{(\lambda, \mu) \in \Omega} I(\lambda, \mu)
=
\frac{2(b^2+1)^{3/2}}{3\sqrt{3}\,b^2}.
\]
\end{proof}

Using Lemma~\ref{main lemma}, we compute the supremum of the Dirichlet energy for a class of embeddings of rectangular tori $T_{0,b}$ into $S^3$. Recall that for a smooth map $F:\big(T_{a,b},g\big) \to S^n$, the Dirichlet energy is defined by
\[
E(F)=\tfrac12\int_{T_{a, b}} |\nabla F|^2\,dv_g.
\]
A class of such maps has energy ratios invariant under conformal transformations of the ambient sphere $S^{2n-1}$; see El~Soufi-Ilias-Ros~\cite{EIR97}.

\begin{proposition}[El~Soufi-Ilias-Ros\textup{\cite{EIR97}}, Proposition 3.1]\label{El-identity}
Let $(A_i)_{1\leq i\leq n}$ be real numbers such that $\sum_{i=1}^{n}A_i^2=1$, and let $(p_i,q_i)_{1\leq i\leq n}$ be pairs of integers. For every $(a,b)\in \mathscr{M}$, consider the map $F_{a,b}:T_{a,b}\to S^{2n-1}$ given by 
\[
F_{a,b}=\Big(A_1f^{a,b}_{p_1q_1},A_1 g^{a,b}_{p_1q_1},\dots, A_nf^{a,b}_{p_nq_n},A_n g^{a,b}_{p_nq_n}\Big).
\]
Then, for every $\phi_\xi\in G(2 n-1)$, every $(a,b)\in \mathscr{M}$ and $(a',b')\in\mathscr{M}$, we have
\[
\frac{E(\phi_\xi\circ F_{a,b})}{E(\phi_\xi\circ F_{a',b'})}=\frac{E(F_{a,b})}{E(F_{a',b'})}.
\]
\end{proposition}

We now compute the supremum of the Dirichlet energy for the embeddings $\Psi_{a,b}: T_{a,b}\to S^3$ defined in~\eqref{embedding}.

\begin{theorem}\label{energy functional}For each $(a,b)\in\mathscr{M}$, let $\Psi_{a,b}:T_{a,b} \to S^3$ be the embedding defined by
\begin{equation}\label{embedding}
\Psi_{a,b}(x,y)=\left(\sqrt{r}\cos{\frac{2\pi y}{b}},\sqrt{r}\sin{\frac{2\pi y}{b}},\sqrt{1-r}\cos{2\pi\left(x-\frac{a y}{b}\right)},\sqrt{1-r}\sin{2\pi\left(x-\frac{a y}{b}\right)}\right),
\end{equation}
where $\tfrac12 \leq r <1$. If $\tfrac12 \leq r\leq\tfrac23$, then
\[
\sup_{\phi_{\xi} \in G(3)} E(\phi_{\xi} \circ \Psi_{a,b}) = \frac{2 \pi ^2 \left((a^2+b^2)(1-r)+r\right)}{b};
\]
and if $\tfrac23<r<1$, then
\[
\sup_{\phi_{\xi}\in G(3)}E(\phi_{\xi}\circ\Psi_{a,b})=\frac{4 \pi ^2  \left((a^2+b^2)(1-r)+r\right)}{3 \sqrt{3} b r\sqrt{1-r}}.          
\]
\end{theorem}

\begin{proof}
Set $b_{0}=\sqrt{\frac{r}{1-r}}$. For the embedding $\Psi_{a,b}$ and the conformal immersion $\psi_{0,b_{0}}$, Proposition~\ref{El-identity} yields 
\[
\frac{E(\phi_{\xi}\circ\Psi_{a,b})}{E(\phi_{\xi}\circ\psi_{0,b_{0}})}=\frac{E(\Psi_{a,b})}{E(\psi_{0,b_{0}})}.
\]
Since $E(f)=A(f)$ for every conformal immersion $f$, we obtain
\[
\frac{E(\phi_{\xi}\circ\Psi_{a,b})}{A(\phi_{\xi}\circ\psi_{0,b_{0}})}=\frac{E(\Psi_{a,b})}{A(\psi_{0,b_{0}})}.
\]
Combining this identity with Lemma~\ref{main lemma} yields the desired formulas.
If $\tfrac12\leq r\leq\tfrac23$, then
\[
\begin{aligned}
\sup_{\phi_{\xi}\in G(3)}E(\phi_{\xi}\circ\Psi_{a,b})&=\sup_{\phi_{\xi}\in G(3)}A(\phi_{\xi}\circ\psi_{0,b_{0}})\cdot \frac{E(\Psi_{a,b})}{A(\psi_{0,b_{0}})}\\
            &=\frac{2 \pi ^2 \left(a^2+b^2+b_0^2\right)}{b \left(b_0^2+1\right)}\\
            &=\frac{2 \pi ^2 \left((a^2+b^2)(1-r)+r\right)}{b}.
\end{aligned}
\]
If $\tfrac23<r<1$, then
\[
\begin{aligned}
\sup_{\phi_{\xi}\in G(3)}E(\phi_{\xi}\circ\Psi_{a,b})&=\sup_{\phi_{\xi}\in G(3)}A(\phi_{\xi}\circ\psi_{0,b_{0}})\cdot \frac{E(\Psi_{a,b})}{A(\psi_{0,b_{0}})}\\
            &=\frac{4 \pi ^2 \sqrt{b_0^2+1} \left(a^2+b^2+b_0^2\right)}{3 \sqrt{3} b b_0^2}\\
            &=\frac{4 \pi ^2  \left((a^2+b^2)(1-r)+r\right)}{3 \sqrt{3} b r\sqrt{1-r}}.\\
\end{aligned}
\]
\end{proof}

\section{Proofs of the main results}\label{Proofs}
We now present the proof of Theorem~\ref{main thm}.

\begin{proof}[\textbf{Proof of Theorem~\ref{main thm}}]
Let $\Psi_{a,b}: T_{a,b}\to S^{3}$ be the embedding given by~\eqref{embedding}. 
Fix $g\in [g_{a,b}]$. 
By Lemma~\ref{conformal-lemma}, there exists a unique point $\xi\in D^4$ such that 
\[
\int_{T_{a,b}}(\phi_{\xi}\circ\Psi_{a,b})\,dv_g=0.
\]

If $\int_{T_{a,b}}(\phi_{\xi}\circ\Psi_{a,b})\cdot f_{1}\,dv_g=0$, then by the variational characterization of $\lambda_{2}(T_{a,b},g)$, for each $i=1,\dots,4$, we have
\[
\lambda_{2}(T_{a,b},g)\int_{T_{a,b}}X_{e_{i}}^{2}\circ\phi_{\xi}\circ\Psi_{a,b}\,dv_g
\leq 
\int_{T_{a,b}}|\nabla(X_{e_{i}}\circ\phi_{\xi}\circ\Psi_{a,b})|^2\,dv_g,
\]
where $\{e_i\}$ is the orthonormal basis of $\mathbb{R}^4$.
Summing over $i$ and using $\sum_{i=1}^{4}X_{e_{i}}^{2}=1$, we obtain
\begin{equation}\label{case1}
\bar\lambda_{2}(T_{a,b},g)
\leq 
\sum_{i=1}^{4}\int_{T_{a,b}}|\nabla(X_{e_{i}}\circ\phi_{\xi}\circ\Psi_{a,b})|^2\,dv_g
=2E(\phi_{\xi}\circ\Psi_{a,b}).
\end{equation}

If $\int_{T_{a,b}}(\phi_{\xi}\circ\Psi_{a,b})\cdot f_{1}\,dv_g\neq 0$, then by Lemma~\ref{main-lemma}, there exists a spherical cap $C$ and a unique $\xi_{C}\in D^{4}$ such that
\[
\int_{T_{a,b}}\phi_{\xi_{C}}\circ F_{C}\circ\phi_{\xi}\circ\Psi_{a,b}\,dv_g=0
\quad\text{and}\quad
\int_{T_{a,b}}(\phi_{\xi_C}\circ F_C\circ\phi_{\xi}\circ\Psi_{a,b})\cdot f_1\,dv_g=0.
\]
By the variational characterization of $\lambda_{2}(T_{a,b},g)$, for each $i=1,\dots,4$, we have
\[
\lambda_{2}(T_{a,b},g)\int_{T_{a,b}}X_{e_{i}}^{2}\circ\phi_{\xi_{C}}\circ F_{C}\circ\phi_{\xi}\circ\Psi_{a,b}\,dv_g
\leq 
\int_{T_{a,b}}|\nabla (X_{e_{i}}\circ\phi_{\xi_{C}}\circ F_{C}\circ\phi_{\xi}\circ\Psi_{a,b})|^2\,dv_g.
\]
Summing over $i$ gives
\begin{equation}\label{text function}
\bar\lambda_{2}(T_{a,b},g)
\leq 
\int_{T_{a,b}}\sum_{i=1}^{4}|\nabla (X_{e_{i}}\circ\phi_{\xi_{C}}\circ F_{C}\circ\phi_{\xi}\circ\Psi_{a,b})|^2\,dv_g.
\end{equation}
Splitting the integral over $(\phi_{\xi}\circ\Psi_{a,b})^{-1}(C)$ and its complement and using the definition of $F_C$, we obtain
\begin{equation}\label{case2}
\begin{aligned}
&\int_{T_{a,b}}\sum_{i=1}^{4}|\nabla (X_{e_{i}}\circ\phi_{\xi_{C}}\circ F_{C}\circ\phi_{\xi}\circ\Psi_{a,b})|^2\,dv_g \\
&=2\int_{(\phi_{\xi}\circ\Psi_{a,b})^{-1}(C)}
\sum_{i=1}^{4}|\nabla (X_{e_{i}}\circ\phi_{\xi_{C}}\circ\phi_{\xi}\circ\Psi_{a,b})|^2\,dv_g \\
&<2\int_{T_{a,b}}
\sum_{i=1}^{4}|\nabla (X_{e_{i}}\circ\phi_{\xi_{C}}\circ\phi_{\xi}\circ\Psi_{a,b})|^2\,dv_g \\
&=4E(\phi_{\xi_C}\circ\phi_{\xi}\circ\Psi_{a,b}).
\end{aligned}
\end{equation}
Combining~\eqref{case1} and~\eqref{case2}, we obtain
\[
\bar\lambda_{2}(T_{a,b},g)
<
4\sup_{\phi\in G(3)}E(\phi\circ\Psi_{a,b}).
\]
By Theorem~\ref{energy functional}, for all $\tfrac12\leq r\leq\tfrac23$,
\[
\bar\lambda_{2}(T_{a,b},g)
<
4\sup_{\phi\in G(3)}E(\phi\circ\Psi_{a,b})
=
\frac{8 \pi ^2 \left((a^2+b^2)(1-r)+r\right)}{b}.
\]
Hence
\begin{equation}\label{case1-ineq}
\bar\lambda_{2}(T_{a,b},g)\leq \inf_{\tfrac12\leq r\leq\tfrac23}\frac{8 \pi ^2 \left((a^2+b^2)(1-r)+r\right)}{b}=\frac{8 \pi ^2 (a^2+b^2+2)}{3 b}.
\end{equation}
Similarly, for $\tfrac23<r<1$,
\begin{equation}\label{inf ineq}
\bar\lambda_{2}(T_{a,b},g)\leq\inf_{\tfrac23\leq r<1}\frac{16 \pi ^2 \left((a^2+b^2)(1-r)+r\right)}{3 \sqrt{3} b r\sqrt{1-r}}.
\end{equation}
Define
\[
F(r)=\frac{16 \pi ^2 \left((a^2+b^2)(1-r)+r\right)}{3 \sqrt{3} b r\sqrt{1-r}},
\quad \tfrac23<r<1.
\]
A direct computation shows
\[
\inf_{\tfrac23\leq r<1}F(r)=F(r_0),
\quad
r_0=\frac{3(a^2+b^2)-\sqrt{(a^2+b^2)(a^2+b^2+8)}}{2(a^2+b^2-1)}.
\]
Thus
\begin{equation}\label{case2-ineq}
\bar\lambda_{2}(T_{a,b},g)\leq F(r_0)
=
\frac{16\pi^2}{3\sqrt{6}\,b}
\frac{\sqrt{2+a^2+b^2+S}}{a^2+b^2+S}\bigl(3(a^2+b^2)+S\bigr),
\end{equation}
where $S=\sqrt{(a^2+b^2)(8+a^2+b^2)}$.

Combining~\eqref{case1-ineq} and~\eqref{case2-ineq}, we obtain
\begin{equation}\label{main ineq}
\bar\lambda_{2}(T_{a,b},g)
\leq
\min\!\Bigl\{\frac{8 \pi ^2 (a^2+b^2+2)}{3 b},\,F(r_0)\Bigr\}
=F(r_0).
\end{equation}

We now show that the inequality in~\eqref{main ineq} is strict.
Suppose to the contrary that equality holds in~\eqref{main ineq}.
Then equality must hold throughout the chain of estimates leading to~\eqref{main ineq};
in particular, we must be in the second case, i.e.,
\begin{equation}\label{neq}
\int_{T_{a,b}}(\phi_{\xi}\circ\Psi_{a,b})\cdot f_{1}\,dv_g\neq 0.
\end{equation}
More precisely, under this assumption there exists a metric 
$g_0=\omega g_{a,b}$ with $\omega>0$ such that 
$\bar\lambda_2(T_{a,b},g_0)=F(r_0)$.
Consequently, equality must hold in~\eqref{text function},~\eqref{case2},~\eqref{inf ineq}, and~\eqref{case2-ineq}.
The corresponding conformal map can therefore be written as $\phi_{\xi_0}$ with
$\xi_0=(\sqrt{3r_0-2},0,0,0)$, and
\[
\Psi_{a,b}
=\Big(\sqrt{r_0}\cos{\tfrac{2\pi y}{b}},\,
\sqrt{r_0}\sin{\tfrac{2\pi y}{b}},\,
\sqrt{1-r_0}\cos{2\pi(x-\tfrac{a y}{b})},\,
\sqrt{1-r_0}\sin{2\pi(x-\tfrac{a y}{b})}\Big).
\]
Consequently,
\[
\bar\lambda_{2}(T_{a,b},g_0)
=
2\int_{(\phi_{\xi}\circ\Psi_{a,b})^{-1}(C)}
\sum_{i=1}^{4}|\nabla X_{e_{i}}\circ\phi_{\xi_{0}}\circ\Psi_{a,b}|^2\,dv_{g_0}
=
2\int_{T_{a,b}}
\sum_{i=1}^{4}|\nabla X_{e_{i}}\circ\phi_{\xi_{0}}\circ\Psi_{a,b}|^2\,dv_{g_0}.
\]
Hence $C=\phi_{\xi}\circ\Psi_{a,b}(T_{a,b})$, so the restriction of $F_C$ to
$\phi_{\xi}\circ\Psi_{a,b}(T_{a,b})$ is $\mathrm{id}$,
which contradicts~\eqref{neq}.
Therefore the inequality in~\eqref{main ineq} is strict.
\end{proof}

Using Theorem~\ref{main thm}, we now prove Corollary~\ref{coro:uniform_bound}.

\begin{proof}[\textbf{Proof of Corollary~\ref{coro:uniform_bound}}]
We show that the right-hand side of~\eqref{main estimate} is increasing in $a$ and decreasing in $b$. Let
\[
U(a,b)=\frac{16\pi^2}{3\sqrt{6}\,b}\frac{\sqrt{2+a^2+b^2+S}}{a^2+b^2+S}\left(3(a^2+b^2)+S \right),
\]
where $S=\sqrt{(a^2+b^2)(8+a^2+b^2)}$.
A direct computation gives
\[
\begin{aligned}
\frac{\partial U(a,b)}{\partial a}=&\frac{128 \pi ^2 a \left(a^2+b^2\right) \left(a^4+2 a^2 b^2+a^2 S+9 a^2+b^4+b^2 S+9 b^2+5 S+8\right)}{3\sqrt{6} b S \left(a^2+b^2+S\right)^2 \sqrt{a^2+b^2+S+2}},\\[6pt]
\frac{\partial U(a,b)}{\partial b}=&\frac{-128 \pi ^2\left(a^2+b^2\right)  }{3\sqrt{6} b^2 S \left(a^2+b^2+S\right)^2 \sqrt{a^2+b^2+S+2}}\left(a^6+2 a^4 b^2+a^4 S+10 a^4+a^2 b^4\right.\\
&\left.+a^2 b^2 S+11 a^2 b^2+6 a^2 S+16 a^2+b^4+b^2 S+8 b^2+2 S\right).
\end{aligned}
\]
Hence $\frac{\partial U(a,b)}{\partial a}>0$ and $\frac{\partial U(a,b)}{\partial b}<0$ for all $(a,b)\in\mathscr{M}$, and therefore, 
\[
U(a,b)\leq U(\tfrac12,\tfrac{\sqrt3}{2})=\frac{16\pi^2}{\sqrt{3}}.
\]
\end{proof}

Lastly, we give the proof of Corollary~\ref{main coro}.

\begin{proof}[\textbf{Proof of Corollary~\ref{main coro}}]
By Proposition~2 of~\cite{K13}, there exists a sequence of connected surfaces, each diffeomorphic to a torus, degenerating to a disconnected surface whose second topological eigenvalue equals
\[
\frac{8\pi^2}{\sqrt{3}}+8\pi.
\]
Since $U(a,b)\leq U(\tfrac12,b)$ for all $(a,b)\in\mathscr{M}$, the condition
\[
U(\tfrac12,b)<\frac{8\pi^2}{\sqrt{3}}+8\pi
\]
implies (by a direct computation, e.g. in \textsc{Mathematica}) that
\[
b\ge 1.76.
\]
Thus, the Kao-Lai-Osting conjecture reduces to proving that for every torus $T_{a,b}$ with $(a,b)\in\mathscr{M}$ and $b\leq1.76$, and every metric $g$ on $T_{a,b}$ conformal to $g_{a,b}$
\[
\bar\lambda_2(T_{a,b}, g)\leq \frac{8\pi^2}{\sqrt{3}}+8\pi.
\]
\end{proof}

\bibliographystyle{plain}
\bibliography{ref}

@misc{B15,
  author = {R. L. Bryant},
  title  = {On the conformal volume of 2-tori},
  year   = {2015},
  eprint = {1507.01485},
  archivePrefix = {arXiv}
}

@article{CE03,
  author  = {B. Colbois and A. El Soufi},
  title   = {Extremal eigenvalues of the Laplacian in a conformal class of metrics: the conformal spectrum},
  journal = {Annals of Global Analysis and Geometry},
  volume  = {24},
  number  = {4},
  pages   = {337--349},
  year    = {2003}
}

@article{CKM19,
  author  = {D. Cianci and M. Karpukhin and V. Medvedev},
  title   = {On branched minimal immersions of surfaces by first eigenfunctions},
  journal = {Annals of Global Analysis and Geometry},
  volume  = {56},
  number  = {4},
  pages   = {667--690},
  year    = {2019}
}

@article{EI02,
  author  = {A. El Soufi and S. Ilias},
  title   = {Extremal metrics for the first eigenvalue of the Laplacian in a conformal class},
  journal = {Proceedings of the American Mathematical Society},
  volume  = {131},
  pages   = {1611--1618},
  year    = {2002}
}

@article{EIR97,
  author  = {A. El Soufi and S. Ilias and A. Ros},
  title   = {Sur la premi\`ere valeur propre des tores},
  journal = {S\'eminaire de Th\'eorie Spectrale et G\'eom\'etrie},
  volume  = {15},
  pages   = {17--23},
  year    = {1997}
}

@misc{EG25,
  author = {M. Eddaoudi and A. Girouard},
  title  = {Upper bounds for the second nonzero eigenvalue of the Laplacian via folding and conformal volume},
  year   = {2025},
  eprint = {2501.08761},
  archivePrefix = {arXiv}
}

@article{EGJ06,
  author  = {A. El Soufi and H. Giacomini and M. Jazar},
  title   = {A unique extremal metric for the least eigenvalue of the Laplacian on the Klein bottle},
  journal = {Duke Mathematical Journal},
  volume  = {135},
  number  = {1},
  pages   = {181--202},
  year    = {2006}
}

@article{GL21,
  author  = {A. Girouard and R. S. Laugesen},
  title   = {Robin spectrum: two disks maximize the third eigenvalue},
  journal = {Indiana University Mathematics Journal},
  volume  = {70},
  pages   = {2711--2742},
  year    = {2021}
}

@article{GNP09,
  author  = {A. Girouard and N. Nadirashvili and I. Polterovich},
  title   = {Maximization of the second positive Neumann eigenvalue for planar domains},
  journal = {Journal of Differential Geometry},
  volume  = {83},
  number  = {3},
  pages   = {637--662},
  year    = {2009}
}

@article{K25,
  author  = {F. Kang},
  title   = {On Berger's isoperimetric problem},
  journal = {Comptes Rendus Math\'ematique},
  volume  = {363},
  pages   = {695--704},
  year    = {2025}
}

@article{H70,
  author  = {J. Hersch},
  title   = {Quatre propri\'et\'es isop\'erim\'etriques de membranes sph\'eriques homog\`enes},
  journal = {Comptes Rendus de l'Acad\'emie des Sciences de Paris},
  volume  = {270},
  pages   = {A1645--A1648},
  year    = {1970}
}

@article{JNP06,
  author  = {D. Jakobson and N. Nadirashvili and I. Polterovich},
  title   = {Extremal metric for the first eigenvalue on a Klein bottle},
  journal = {Canadian Journal of Mathematics},
  volume  = {58},
  number  = {2},
  pages   = {381--400},
  year    = {2006}
}

@article{K13,
  author  = {M. A. Karpukhin},
  title   = {Nonmaximality of known extremal metrics on torus and Klein bottle},
  journal = {Matematicheskii Sbornik},
  volume  = {204},
  number  = {12},
  pages   = {31--48},
  year    = {2013}
}

@article{K93,
  author  = {N. Korevaar},
  title   = {Upper bounds for eigenvalues of conformal metrics},
  journal = {Journal of Differential Geometry},
  volume  = {37},
  pages   = {73--93},
  year    = {1993}
}

@article{K21,
  author  = {M. Karpukhin},
  title   = {Index of minimal spheres and isoperimetric eigenvalue inequalities},
  journal = {Inventiones Mathematicae},
  volume  = {223},
  pages   = {335--377},
  year    = {2021}
}

@article{K22,
  author  = {H. N. Kim},
  title   = {Maximization of the second Laplacian eigenvalue on the sphere},
  journal = {Proceedings of the American Mathematical Society},
  volume  = {150},
  number  = {8},
  pages   = {3501--3512},
  year    = {2022}
}

@article{KLO17,
  author  = {C. Y. Kao and R. Lai and B. Osting},
  title   = {Maximization of Laplace--Beltrami eigenvalues on closed Riemannian surfaces},
  journal = {ESAIM: Control, Optimisation and Calculus of Variations},
  volume  = {23},
  pages   = {685--720},
  year    = {2017}
}

@article{KS24,
  author  = {M. Karpukhin and D. Stern},
  title   = {Min-max harmonic maps and a new characterization of conformal eigenvalues},
  journal = {Journal of the European Mathematical Society},
  volume  = {26},
  pages   = {4071--4129},
  year    = {2024}
}

@article{KNPP17,
  author  = {M. Karpukhin and N. Nadirashvili and A. Penskoi and I. Polterovich},
  title   = {An isoperimetric inequality for Laplace eigenvalues on the sphere},
  journal = {Journal of Differential Geometry},
  volume  = {118},
  number  = {2},
  pages   = {313--333},
  year    = {2021}
}

@article{L21,
  author  = {R. S. Laugesen},
 title = {Well-posedness of Hersch--Szeg{\'o}'s center of mass by hyperbolic energy minimization},
  journal = {Annales math\'ematiques du Qu\'ebec},
  volume  = {45},
  number  = {2},
  pages   = {363--390},
  year    = {2021}
}

@article{LY82,
  author  = {P. Li and S.-T. Yau},
  title   = {A new conformal invariant and its applications to the Willmore conjecture and the first eigenvalue of compact surfaces},
  journal = {Inventiones Mathematicae},
  volume  = {69},
  number  = {2},
  pages   = {269--291},
  year    = {1982}
}

@article{MR86,
  author  = {S. Montiel and A. Ros},
  title   = {Minimal immersions of surfaces by the first eigenfunctions and conformal area},
  journal = {Inventiones Mathematicae},
  volume  = {83},
  pages   = {153--166},
  year    = {1986}
}

@article{N96,
  author  = {N. Nadirashvili},
  title   = {Berger's isoperimetric problem and minimal immersions of surfaces},
  journal = {Geometric and Functional Analysis},
  volume  = {6},
  number  = {5},
  pages   = {877--897},
  year    = {1996}
}

@article{N02,
  author  = {N. Nadirashvili},
  title   = {Isoperimetric inequality for the second eigenvalue of a sphere},
  journal = {Journal of Differential Geometry},
  volume  = {61},
  pages   = {335--340},
  year    = {2002}
}

@article{NS19,
  author  = {S. Nayatani and T. Shoda},
  title   = {Metrics on a closed surface of genus two which maximize the first eigenvalue of the Laplacian},
  journal = {Comptes Rendus Math\'ematique},
  volume  = {357},
  number  = {1},
  pages   = {84--98},
  year    = {2019}
}

@article{P14,
  author  = {R. Petrides},
  title   = {Maximization of the second conformal eigenvalue of spheres},
  journal = {Proceedings of the American Mathematical Society},
  volume  = {142},
  number  = {7},
  pages   = {2385--2394},
  year    = {2014}
}

@article{YY80,
  author  = {P. C. Yang and S.-T. Yau},
  title   = {Eigenvalues of the Laplacian of compact Riemann surfaces and minimal submanifolds},
  journal = {Annali della Scuola Normale Superiore di Pisa},
  volume  = {7},
  number  = {4},
  pages   = {55--63},
  year    = {1980}
}
\end{document}